\documentclass[11pt]{amsart}
\usepackage{mathrsfs}%»¨Ìå×Öĸ
\usepackage{eufrak}
\usepackage{amssymb}
\usepackage{graphicx}

\usepackage[colorlinks,linkcolor=blue,citecolor=blue,hypertexnames=false]{hyperref}%colorful links

\DeclareFontFamily{OT1}{rsfs}{}
\DeclareFontShape{OT1}{rsfs}{n}{it}{<-> rsfs10}{}
\DeclareMathAlphabet{\curly}{OT1}{rsfs}{n}{it}

\makeatletter
\newcommand{\eqnum}{\refstepcounter{equation}\textup{\tagform@{\theequation}}}
\makeatother

\newtheorem{thm}{Theorem}[section]
\newtheorem{prop}[thm]{Proposition}

\newtheorem{lem}[thm]{Lemma}
{\theoremstyle{definition}

\newtheorem{rem}[thm]{Remark}}

\newcommand{\hl}{\mathcal{H}}

\newcommand{\cl}{\mathcal{C}}

\newcommand{\ml}{\mathcal{M}}
\newcommand{\nl}{\mathcal{N}}
\newcommand{\fl}{\mathcal{F}}

\newcommand{\ol}{\mathcal{O}}

\newcommand{\ql}{\mathcal{Q}}
\newcommand{\ssk}{\mathfrak{s}}

\newcommand{\rb}{\mathbb{R}}

\newcommand{~}{\quad}
\newcommand{\cb}{\mathbb{C}}

\newcommand{\zb}{\mathbb{Z}}

\newcommand{\undl}{\underline}

\newcommand{\wt}{\widetilde}

\newcommand{\pic}{\mathrm{Pic}}

\begin{document} %\baselineskip 15pt

\title[Remark on GW-W invariants]
{A remark on Gromov-Witten-Welschinger invariants of $\cb P^3\#\overline{\cb P}^3$}

\author{Yanqiao Ding}%$^{*}$}

\address{School of Mathematics and Statistics\\ Zhengzhou University\\
                       Zhengzhou, 450001 \\ P. R. China}
\email{yqding@zzu.edu.cn}

%\thanks{${}^{*}$Partially supported by Startup Research Fund of Zhengzhou University (No. 161131003).}


\begin{abstract}
  We generalize the formula of Gromov-Witten-Welschinger
  invariants of $\cb P^3$ established by E. Brugall\'e and P. Georgieva in \cite{bg2016}
  to $\cb P^3\#\overline{\cb P}^3$. Using pencils of quadrics, some real and complex
  enumerative invariants of $\cb P^3\#\overline{\cb P}^3$ can be written as the combination
  of the enumerative invariants of the blow up of $\cb P^2$ at two real points.
\end{abstract}

\maketitle

\section{Introduction}
It is a long history for mathematicians to find appropriate way to count curves.
Almost 24 years ago, inspired by string theory \cite{w1991}, Y. Ruan and G. Tian
introduced the definition of Gromov-Witten invariants on semi-positive
symplectic manifolds \cite{rt1995}.
Then the definition of Gromov-Witten invariants was generalized to symplectic manifolds
\cite{ruan1999,fo1999,lt1998} and algebraic varieties \cite{lt1998a,beh1997}.
With the help of Gromov-Witten invariants, complex enumerative
geometry experiences a rapid development. The emergence of real enumerative
invariants, such as Welschinger invariants \cite{wel2005a,wel2005b},
real Gromov-Witten invariants \cite{gz2018}, and refined tropical enumerative
invariants \cite{bg2016a}, also prompts the development of real enumerative geometry
significantly.

Welschinger invariant is a sign counting of real curves in real rational surfaces or
real convex algebraic threefolds. There are many methods, for example, tropical geometry and
degeneration formula, to investigate the Welschinger invariants of real rational surfaces
\cite{iks2009,iks2013,iks2013a,bp2014,brugalle2016,bm2007,bru2015,brugalle2016,dh2016a,ding2017}.
However, the methods for computation of
Welschinger invariants of threefolds are limited.
For instance, some particular cases of Welschinger invariants of threefolds
were calculated in \cite{bm2007,gz2013,wel2007,psw2008}.
%when there are only real point constraints \cite{bm2007},
%when the number of real point constraints is minimal \cite{gz2013}, \cite{wel2007},
%and when there is no constraints \cite{psw2008}.
Koll\'ar proposed pencils of quadrics of $\cb P^3$ can be used
to compute the enumerative invariants of $\cb P^3$ \cite{k2014}.
E. Brugall\'e and P. Georgieva completely computed the Welschinger invariants
of $\cb P^3$ \cite{bg2016}. They estiblished a relation between the GWW invariants of
$\cb P^3$ and the GWW invariants of $\cb P^1\times\cb P^1$.
In this note, we generalize this relation to $\cb P^3\#\overline{\cb P}^3$.

Equip $\cb P^3$ and $\cb P^1$ with the standard real structure which is the complex conjugation.
Denote by $Y=\cb P^3\#\overline{\cb P}^3$ the blow-up of $\cb P^3$ at a real point $x_0$,
by $l$ the strict transform of a line in $\cb P^3$ which does not pass through $x_0$,
by $E$ the exceptional divisor, and by $Q=\cb P^1\times\cb P^1$
a quadric surface in $\cb P^3$ passing through $x_0$.
According to the classification results of real surfaces \cite{kollar97,dk00,dk02},
the blow-up at a real point $x_0$ of the quadric surface
$\wt Q=(\cb P^1\times\cb P^1)_{1,0}$ is real deformation equivalent to
the blow up of $\cb P^2$ at two real points $\cb P^2_{2,0}$.
Let $l_1$, $l_2$ be the strict transforms of two lines,
which are $\cb P^1\times\{p_0\}$ and $\{p_0\}\times\cb P^1$ up to interchanging with each other,
in $\cb P^1\times\cb P^1$ not passing $x_0$.
The homology group $H_2((\cb P^1\times\cb P^1)_{1,0};\zb)$ is generated by the
classes $[l_1]$, $[l_2]$ and $[E_{\wt Q}]$, where $E_{\wt Q}$
is the exceptional divisor of $(\cb P^1\times\cb P^1)_{1,0}$.
We use $(a,b,c)$ as an abbreviation of the class $a [l_1]+b [l_2]-c[E_{\wt Q}]$
in $H_2((\cb P^1\times\cb P^1)_{1,0};\zb)$.

Let $d$ and $k$ be two positive integers with $d\geq k$. Given a generic real
configuration $\undl x$ of $2d-k$ points in $\cb P^3\#\overline{\cb P}^3$
which contains $s$ pairs of complex conjugated points.
Denote by $W_{\cb P^3\#\overline{\cb P}^3}(d[l]-k[e],s)$ the Welschinger
invariant ``counting" the real rational curves passing through $\undl x$ and
representing the class $d[l]-k[e]$ in $\cb P^3\#\overline{\cb P}^3$,
where $e$ is a line in the exceptional divisor $E$.
Denote by $W_{\wt Q}((a,b,k),s)$ the Welschinger invariant ``counting"
the real rational curves passing through $\undl y$, which is a generic real
configuration of $2(a+b)-k-1$ points in $\wt Q$ with $s$ pairs of complex conjugated points,
and representing the class $(a,b,c)$ in $\wt Q$.

\begin{thm}\label{thm1}
For any $d\in\zb_{>0}-2\zb$, $k\in2\zb_{\geq0}$ with $d\geq k$, and $s\in\{0,\ldots,d-\frac{k}2-1\}$,
we have
\begin{equation}\label{sec1-eq-thm1}
W_{\cb P^3\#\overline{\cb P}^3}(d[l]-k[e],s)=\sum_{\substack{a+b=d \\ 0\leq a<b}}
(-1)^{a+\frac{k}2}(d-2a)W_{\wt Q}((a,b,k),s).
\end{equation}
\end{thm}

By Remark $\ref{subsec4.1-rem-2}$, $W_{\cb P^3\#\overline{\cb P}^3}(d[l]-k[e],s)=0$,
where $d$, $k$ are positive even integers and $s\leq d-\frac{k}2-1$.
Since Welschinger invariants of any blow-up of the real projective plane and the real quadric $Q$ in $\cb P^3$
are computed by Horev and Solomon in \cite{hs2012},
one can use equation ($\ref{sec1-eq-thm1}$) to determine all the Welschinger invariants
$W_{\cb P^3\#\overline{\cb P}^3}(d[l]-k[e],s)$ with $k$ is even.

A similar result of Gromov-Witten invariants can also be derived.
Denote by $\langle[pt],\ldots,[pt]\rangle^{\cb P^3\#\overline{\cb P}^3}_{0,d[l]-k[e]}$
the Gromov-Witten invariant counting the rational curves in $\cb P^3\#\overline{\cb P}^3$
passing through $2d-k$ generic points and representing the class $d[l]-k[e]$.
%where $e$ is a line in the exceptional divisor of $\cb P^3\#\overline{\cb P}^3$.
Denote by $\langle[pt],\ldots,[pt]\rangle^{\wt Q}_{0,(a,b,k)}$ the Gromov-Witten invariant
counting the rational curves in $\wt Q$ passing through $2(a+b)-k-1$ generic points
and representing the class $(a,b,c)$.

\begin{thm}\label{thm2}
Let $d\in\zb_{>0}$, $k\in\zb_{\geq0}$ and $d\geq k$. Then
\begin{equation}\label{sec1-eq-thm2}
\langle[pt],\ldots,[pt]\rangle^{\cb P^3\#\overline{\cb P}^3}_{0,d[l]-k[e]}=
\sum_{\substack{a+b=d \\ 0\leq a<b}}(d-2a)^2\langle[pt],\ldots,[pt]\rangle^{\wt Q}_{0,(a,b,k)}.
\end{equation}
\end{thm}

\medskip\noindent
{\bf Acknowledgments}: The author would like to thank Jianxun Hu
for his continuous support and encouragement as well as enlightening discussions.
%The research is partially supported by Startup Research Fund of Zhengzhou University (No. 161131003).
\section{Welschinger invariants}

\subsection{Welschinger invariants of  $\cb P^2_{2,0}$}

Given $a,b,c\in\zb_{\geq0}$ with $2(a+b)-c-1>0$.
Let $\wt Q=(\cb P^1\times\cb P^1)_{1,0}$ which is real deformation equivalent to $\cb P^2_{2,0}$
and $\undl x\subset \wt Q$ be a real configuration of $2(a+b)-c-1$ points
with $s$ pairs of complex conjugated points.
Denote by $\cl(\undl x)$ the set of real rational curves representing
the class $(a,b,c)$ and passing through $\undl x$ in $\wt Q$.
If $\undl x$ is generic enough, the set $\cl(\undl x)$
is finite, and every curve $C\in\cl(\undl x)$ is an immersion with only nodal points.
Denote by $m_{\wt Q}(C)$ the number of real isolated nodes
(nodes with two complex conjugated branches) of $C$.
The integer
$$
W_{\wt Q}((a,b,c),s)=\sum_{C\in\cl(\undl x)}(-1)^{m_{\wt Q}(C)}
$$
is independant of the choice of the real configuration $\undl x$,
and it only depends on $a$, $b$, $c$, and the number $s$ of pairs
of complex conjugated points in $\undl x$ (see \cite{wel2005a}).

%Let $C\subset X$ be an irreducible real rational curve such that $[C]=(a,b,2c)$.
%Assume $X$ contains a real point $x_0\in\rb C$, and fix an orientation on $T_{x_0}\rb X$.
%This orientation will induce an orientation on $T\rb X|_{\rb C}$ as
%$w_1(\rb X)\cdot[\rb C]=c_1(X)\cdot[C]=0 \mod 2$.
%Fix a Riemannian metric $g$ on the real part $\rb X$ of $X$.
%Equip an orientation on $\rb C$, and let $N_{\rb C}$ be the
%orthogonal subbundle of $\rb C$ in $\rb X$. Choose a nowhere vanishing
%smooth section $\sigma_1\in\Gamma(T\rb C)$ and a section $\sigma_2\in\Gamma(N_{\rb C})$
%such that $(\sigma_1,\sigma_2)$ is a positive basis of $T\rb X|_{\rb C}$.
%Let $N$ be the parity of the number of times the loop $(\sigma_1(x),\sigma_2(x))_{x\in\rb C}$
%rotates around the canonical basis of $T\rb X|_{\rb C}$.
%
%\begin{lem}\label{sec2-lem-1}
%Let $C$ be an irreducible real rational curve in $X$ such that $[C]=(a,b,2c)$,
%then $m_X(C)=N+c \mod 2$.
%\end{lem}
%\begin{proof}
%The proof is similar to the proof of \cite[Lemma 2.1]{bg2016}.
%Equip an orientation on $\rb C$, and smooth each node of $\rb C$ as depicted in Figure.
%We can get a collection $\gamma$ of oriented circles embedded in $\rb X$.
%Then $N$ equals the sum of the number of times the loop $(\sigma_1(x),\sigma_2(x))_{x\in\gamma}$
%rotates around the canonical basis of $T\rb X|_{\gamma}$.
%\end{proof}

\subsection{Welschinger invariants of  $\cb P^3\#\overline{\cb P}^3$}\label{subsec2.2}
%Let $Y$ be the blow-up of $\cb P^3$ at a real point $x_0$,
%and $E\subset Y$ be the exceptional divisor.
%Denote by $e$ a line in the exceptional divisor $E$,
%and by $l$ the strict transform of a line in $\cb P^3$ which does not pass through $x_0$.

Given $d,k\in\zb_{\geq0}$ with $d\geq k$, and let $\undl x$ be a generic real configuration,
which contains $s$ pairs of complex conjugated points and at least one real point,
of $2d-k$ points in $Y$. Denote by $\cl'(\undl x)$
the set of irreducible real rational curves in $Y$ of class $d[l]-k[e]$ passing through $\undl x$.
The set $\cl'(\undl x)$ is finite, and every element $f:\cb P^1\to Y$
of $\cl'(\undl x)$ is a balanced curve, i.e. $f$ is an immersion and
the quotient $N_f=f^*TY/T\cb P^1$ is isomorphic to the holomorphic bundle
$\ol_{\cb P^1}(2d-k-1)\oplus\ol_{\cb P^1}(2d-k-1)$.
Fix a spin structure $\ssk$ on $\rb P^3$ and a spin structure
$\bar\ssk$ on $\overline{\rb P}^3$. They can induce a spin structure
$\ssk\#\bar\ssk$ on $\rb Y=\rb P^3\#\overline{\rb P}^3$.
Then the integer
$$
W_Y(d[l]-k[e],s)=\sum_{C\in\cl'(\undl x)}(-1)^{s_Y(C)}
$$
is independant of the choice of $\undl x$, where $s_Y(C)$ is the spinor state of $C$.
It depends only on the class $d[l]-k[e]$, the number $s$ of complex conjugated
pairs in $\undl x$ and the spin structure $\ssk\#\bar\ssk$ (see \cite{wel2005b}).
Since we only deal with the case when $k$ is even in this note,
we always assume $k$ is even. The spinor state $s_Y(C)$ is defined as follows.

Let $(f:\cb P^1\to Y)\in\cl'(\undl x)$ be a balanced real algebraic immersion.
Fix an orientation on $\rb Y=\rb P^3\#\overline{\rb P}^3$ and equip $\rb Y$
with a Riemannian metric $g$. Choose an orientation on $f(\rb P^1)$ which is
induced by the orientation on $\rb P^1$. The tangent bundle $T\rb P^1$ is a
subbundle of $f^*_{|\rb P^1}T\rb Y$. The real vector bundle $\rb N_f$,
which is isomorphic to the orthogonal subbundle of $T\rb P^1$ in $f^*_{|\rb P^1}T\rb Y$,
has an induced orientation such that a positive orthonormal frame of $T_x\rb P^1$ followed by a
positive orthonormal frame of $\rb N_f|_x$ provides a positive orthonormal frame of $f_x^*T\rb Y$
for every $x\in\rb P^1$. Denote by $H$ the class of the real part of a section of
$P(N_f)\cong\cb P^1\times\cb P^1$ having vanishing self-intersection,
by $F$ the class of a real fiber of $P(\rb N_f)$.
The orientations on $H$, and $F$ are induced by the
orientations on $\rb P^1$ and $\rb N_f$ respectively. Choose a real holomorphic line subbundle
$L_f\subset N_f$ such that $P(L_f)\subset P(N_f)$ is a section of bidegree $(1,1)$ whose
real part is homologous to $\pm(H+F)\in H_1(\rb N_f;\zb)$. The bundle $T\rb P^1\oplus\rb L_f$
is equipped with a Riemannian metric induced by the one of $T\rb Y$. Choose a loop
of positive orthonormal frames $(\sigma_1(x),\sigma_2(x))_{x\in\rb P^1}$ of $T\rb P^1\oplus\rb L_f$
such that $(\sigma_1(x))_{x\in\rb P^1}$ is a loop of positive orthonormal frames of $T\rb P^1$.
Let $(\sigma_3(x))_{x\in\rb P^1}$ be the unique section of $f^*_{|\rb P^1}T\rb Y$
such that $(\sigma_1(x),\sigma_2(x),\sigma_3(x))_{x\in\rb P^1}$ is a loop of positive orthonormal frames
of $f^*_{|\rb P^1}T\rb Y$. Define $S_Y(f)=0$ if this loop of the $SO_3(\rb)$-principal bundle
$R_Y$ of orthonormal frames of $T\rb Y$ lifts to a loop of the associated $Spin$-principal bundle $P_Y$,
and $S_Y(f)=1$ otherwise. Note the spinor state depends only on the spin structure $\ssk\#\bar\ssk$
and the isotopy class of $\rb L_f$.

A more geometric way to explain the construction of the holomorphic line bundle $L_f$ is as follows.
The choice of the class $\pm(H+F)$ is equivalent to choice a direction
to rotate the ribbon of $f(\rb P^1)$ provided by the real holomorphic section $H\in P(N_f)$
such that the ribbon after rotated is orientable. The isotopy class of $\rb L_f$ is the isotopy
class of the real part of the degree $2d-k-2$ line subbundle of $N_f$ whose real fibers rotate positively
with respect to the positive basis of $\rb N_f$ in local holomorphic coordinates on $N_f$ defining a
real holomorphic splitting $N_f=\ol_{\cb P^1}(2d-k-1)\oplus\ol_{\cb P^1}(2d-k-1)$
(see \cite[Section 2.2]{bg2016} and \cite[Section 2.2]{wel2005b} for more details).

\section{Rational curves on a blow-up of smooth quadric of $\cb P^3$}

\subsection{Elliptic curves, pencils of quadrics and blow-up}
Let $C$ be a complex elliptic curve.
The choice of a point $p$ in $C$ will induce an isomorphism $\Phi$ between the elliptic curve $C$ and the
group $\pic_0(C)$. This isomorphism gives a group structure on $C$.
If $C$ is considered as the quotient of $\cb$ by a full rank lattice $\Lambda$
for which $p$ is identified with $0$,
the group structure induced by $\Phi$ is the same with the group structure
induced from $(\cb,+)$ by the quotient map.
$\Phi$ also induces a series of isomorphisms $\Phi_d$ between $\pic_d(C)$ and $C$
(see \cite[Section 3.1]{bg2016} for more details). In this note,
$\pic_d(C)$ is always regarded as $C$ under the isomorphism $\Phi_d$.
If $C$ is also real, with $\rb C\neq\emptyset$, and $p\in\rb C$,
there are two cases as follows (see \cite[Section 3.1]{bg2016}):
\begin{itemize}
  \item $\rb C$ has two connected components: if $d$ is even, both components contain $d$
  torsion points of order $d$ on $C$; if $d$ is odd, the connected component of $\rb C$ containing
  $p$ contains $d$ torsion points of order $d$;
  \item $\rb C$ is connected, and $\rb C$ contains $d$ torsion points of order $d$;
\end{itemize}

Let $C_4$ be a non-degenerate elliptic curve of degree $4$ in $\cb P^3$,
and $\ql$ be the pencil of quadrics in $\cb P^3$ % which is also a line in $\cb P^9$
determined by $C_4$.
Let $h\in C_4\simeq\pic_4(C_4)$ be the hyperplane section class.
There is a ramified covering map $\pi_\ql:C_4\simeq\pic_2(C_4)\to\ql$ of degree two.
From analysing the ramification value of $\pi$, one can get $\ql$ contains $4$ singular quadrics.
If $C_4$ is real and $\rb C_4\neq\emptyset$, there are three cases as follows (see \cite[Section 3.2]{bg2016}):
\begin{itemize}
  \item $\rb C_4$ is not connected, and $h$, $p$ are not on the same component:
  there is no real singular quadric in $\ql$;
  \item $\rb C_4$ is not connected, and $h$, $p$ are on the same component:
  there are $4$ real singular quadrics in $\ql$;
  \item $\rb C_4$ is connected: there are $2$ real singular quadrics in $\ql$;
\end{itemize}

Let $\pi:Y=\cb P^3\#\overline{\cb P}^3\to\cb P^3$ be the blow-up of $\cb P^3$ at $x_0\in C_4$.
Denote by $\widetilde C_4\subset Y$ and $\widetilde\ql\subset Y$ the strict transforms of
$C_4$ and $\ql$ respectively.
Denote by $H$ the strict transform of a hyperplane in $\cb P^3$ which does not pass $x_0$,
and by $\tilde h=H|_{\widetilde C_4}\in\pic_4(\widetilde C_4)$ the hyperplane section class.
The map $\pi_\ql:C_4\simeq\pic_2(C_4)\to\ql$ induces a ramified covering map
$\pi_{\wt\ql}:\wt C_4\simeq\pic_2(\wt C_4)\to\wt\ql$ of degree two,
where $\pi_{\wt\ql}(x)$, $x\in\wt C_4$, is the strict transform of
$\pi_\ql\circ\pi_{|\wt C_4}(x)$.
If $C_4$ is real, with $\rb C_4\neq\emptyset$, and $p,x_0\in \rb C_4$, there are three cases as follows:
\begin{itemize}
  \item $\rb \wt C_4$ is not connected, and $\tilde h$, $p$ are not on the same component:
  there is no blow-up of real singular quadric in $\wt\ql$;
  \item $\rb \wt C_4$ is not connected, and $\tilde h$, $p$ are on the same component:
  there are $4$ blow-up of real singular quadrics in $\wt\ql$;
  \item $\rb \wt C_4$ is connected: there are $2$ blow-up of real singular quadrics in $\wt\ql$;
\end{itemize}

\subsection{Rational curves on blow-up of smooth quadrics}

Let $\undl x\subset C_4\setminus x_0$ be a configuration of $2d-k$ distinct points,
%Let $Y=\cb P^3\#\overline{\cb P}^3$ be the blow-up of $\cb P^3$ at $x\in C_4\setminus\undl x$.
%Denote by $l\subset Y$ the strict transform of a line in $\cb P^3$ which does not pass $x_0$.
%Let $e$ be a line in the exceptional divisor $E$ of $Y$.
%Let $E_{\widetilde Q}$ be the exceptional divisor of
%the blow-up of a quadric $\widetilde Q$ of $\widetilde\ql$.
and $e_{\widetilde Q}=E_{\widetilde Q}|_{\widetilde C_4}$.
Denote by $\overline\cl'(\undl x)$ the set of connected algebraic curves of arithmetic genus 0
representing $d[l]-k[e]$ in $Y$ and passing through $\undl x$.

\begin{lem}\label{subsec3.1-lem-1}
Suppose the points in $\undl x$ are in general position in $C_4\setminus x_0$.
Then every curve $C$ in $\overline\cl'(\undl x)$ is irreducible and is contained
in a blow-up of quadric of $\widetilde\ql$. Furthermore,
the blow-up of quadrics of $\widetilde\ql$ which contain a curve in $\overline\cl'(\undl x)$
correspond to the solutions $L_2\in\pic_2(\widetilde C_4)$ of the equation
\begin{equation}\label{subsec3.1-eq-1}
(d-2a)L_2=(d-a)\tilde h-ke_{\widetilde Q}-\undl x,
\end{equation}
with $0\leq a<\frac{d}{2}$.

A curve $C$ in $\overline\cl'(\undl x)$ is linearly equivalent in such a blow-up of quadric
$\widetilde Q$ to $al_1+(d-a)l_2-kE_{\widetilde Q}$,
where $l_1$ (resp. $l_2$) is a line in $\wt Q$ whose intersection with
$\wt C_4$ is $L_2$ (resp. $\tilde h-L_2$). Any irreducible rational curve $C$ in $\wt Q$
representing $al_1+(d-a)l_2-kE_{\widetilde Q}$ and passing through $2d-k-1$ points
of $\undl x$ is in $\overline\cl'(\undl x)$.
\end{lem}

\begin{proof}
The proof is similar to the proof of \cite[Proposition 3]{k2014}.
For the completeness, we give it in the following.
Suppose the curve $C\in\overline\cl'(\undl x)$ is irreducible.
Since any point of $C\setminus\undl x$ is contained in some $\wt Q$,
$C$ and $\wt Q$ intersect at $\geq 2d-k+1$ points which implies that $C\subset\wt Q$.
Hence $C$ is linearly equivalent to $al_1+(d-a)l_2-kE_{\widetilde Q}$ for some $0\leq a\leq\frac{d}{2}$.
So
\begin{equation}\label{subsec3.1-eq-2}
\begin{aligned}
\undl x=(\wt C_4\cdot C)&\sim a L_2+(d-a)(\tilde h-L_2)-ke_{\widetilde Q}\\
&=(d-a)\tilde h-(d-2a)L_2-ke_{\widetilde Q}.
\end{aligned}
\end{equation}
For the points in $\undl x$ are in general position, the case $a=d/2$ is excluded
and $\wt Q$ is the blow-up of a smooth quadric.

In the case $C\in\overline\cl'(\undl x)$ is reducible,
let $\Sigma C^i$ be a connected curve passing through $\undl x$ representing $d[l]-k[e]$.
If there is a component $C^i$ passes through more than $1/2c_1(Y)\cdot[C^i]$ points of $\undl x$,
it intersects with every $\wt Q$ of $\wt\ql$ at $>1/2c_1(Y)\cdot[C^i]$ points.
Therefore, $C^i=\wt C_4$ which is a contradiction.
So every component $C^i$ passes through exactly $1/2c_1(Y)\cdot[C^i]$ points $\undl x^i\subset\undl x$.
Since $\Sigma[C^i]=d[l]-k[e]$, $\undl x$ is the union of the disjoint sets $\undl x^i$.

If there are two curves $C^i$, $C^j$ contained in different blow-up of quadrics
$\wt Q^i\neq\wt Q^j$, we have $C^i\cap C^j\subset\wt C_4$. However,
we know different curves $C^i$ pass through different points $\undl x^i$.
The curve $\Sigma C^i$ has to be disconnected which is a contradiction.
Therefore, the curve $\Sigma C^i$ is contained in one blow-up of quadric $\wt Q$.

Suppose $C^i$ is linearly equivalent to $a^il_1+(d^i-a^i)l_2-k^ie_{\wt Q}$,
we have
$$
\undl x^i\sim (d^i-a^i)\tilde h-(d^i-2a^i)L_2-k^ie_{\widetilde Q}.
$$
Combined with equation $(\ref{subsec3.1-eq-2})$, we obtain
\begin{equation}\label{subsec3.1-eq-3}
\begin{aligned}
(d-2a)\undl x^i\sim &(d^i-2a^i)\undl x +[(d-2a)(d^i-a^i)-(d^i-2a^i)(d-a)]\tilde h\\
&+[(d^i-2a^i)k-(d-2a)k^i]e_{\widetilde Q}.
\end{aligned}
\end{equation}
We can choose the points of $\undl x$ such that $\undl x$, $\tilde h$ and
$e_{\widetilde Q}$ generate a free subgroup of rank $2d-k+2$ in $\pic(\wt C_4)$.
So equation $(\ref{subsec3.1-eq-3})$ is impossible unless $\undl x^i=\undl x$.
Therefore, there is no reducible curve $C$ for a general choice of the points $\undl x$.
\end{proof}

Given a quadric $Q$ of $\cb P^3$, and let $f:\cb P^1\to \cb P^3$
be an algebraic immersion whose image is contained in $Q$ passing through $x_0$.
Suppose $z=f^{-1}(x_0)\subset\cb P^1$.
Let $N_f'=f^*TQ/T\cb P^1$, $N_f=f^*T\cb P^3/T\cb P^1$, and $N_{Q}=T\cb P^3|_{Q}/TQ$
be the normal bundle of $Q$ in $\cb P^3$.
Let $\tilde f:\cb P^1\to Y$ be the strict transform of $f$. Denote by
$N_{\tilde f}'=\tilde f^*T\wt Q/T\cb P^1$, $N_{\tilde f}=\tilde f^*T Y/T\cb P^1$,
and $N_{\wt Q}=TY|_{\wt Q}/T\wt Q$ the normal bundle of $\wt Q$ in $Y$.
We can get an exact sequence of holomorphic vector bundles over $\cb P^1$:
\begin{equation}\label{subsec3.1-eq-4}
0\to N_{\tilde f}'\to N_{\tilde f}\to \tilde f^*N_{\wt Q}\to0.
\end{equation}
If $Q$ and $f$ are both real, we can also obtain an exact sequence fitted by the real bundles
$\rb N_{\tilde f}'$, $\rb N_{\tilde f}$ and $\rb N_{\wt Q}$.

The morphism $\pi:Y\to\cb P^3$ induces morphisms between the bundles
associated to $\tilde f$ and $f$ which are isomorphisms everywhere except $z$ and vanish at $z$.
Since $f$ is an immersion, one can see the vanishing order is $1$ in local coordinates.
Therefore, $N_{\tilde f}'=N_{f}'\otimes\ol_{\cb P^1}(-z)$,
$N_{\tilde f}=N_{f}\otimes\ol_{\cb P^1}(-z)$, and
$\tilde f^*N_{\wt Q}=f^*N_{Q}\otimes\ol_{\cb P^1}(-z)$.
It follows that if $f$ is balanced, so is $\tilde f$.
From \cite[Proposition 4.1]{bg2016},
we know that if $f$ has bidegree $(a,b)$ with $a\neq b$, then $f$ is balanced.
So the strict transform of the algebraic immersion $f:\cb P^1\to \cb P^3$
which is contained in $Q$ with bidegree $(a,b)$, $a\neq b$, passing through $x_0$ is also balanced.
%\begin{prop}[{\cite[Proposition 4.1]{bg2016}}]\label{subsec3.1-prop-1}
%Let $f:\cb\to\cb P^3$ be an algebraic immersion such that $f(\cb P^1)$ is contained in $Q$,
%where it has bidegree $(a,b)$ with $a\neq b$, then $f$ is balanced.
%\end{prop}

Let $\rb N_i'$ be the real part of the normal bundle of $l_i$ in $Q$
which is a subbundle of the real part of its normal bundle in $\cb P^3$.
From \cite[Proposition 4.2]{bg2016}, one has $s_{\cb P^3}(\rb N_1')\neq s_{\cb P^3}(\rb N_2')$.
Following \cite{bg2016}, $(l_1,l_2)$ is called a
\textit{positive basis} of $H_2(Q;\zb)$ if $s_{\cb P^3}(\rb N_1')=0$.

\begin{prop}\label{subsec3.1-prop-2}
Let $f:\cb P^1\to \cb P^3$ be a real algebraic immersion, whose image is contained in $Q$,
passing through $x_0$ such that it has bidegree $(a,b)$ in the positive basis with $a\neq b$.
Suppose $z=f^{-1}(x_0)\subset\cb P^1$ contains $k$ points, and $k$ is even.
Let $\tilde f$ be the strict transform of $f$ in $Y$.
Then the holomorphic real line subbundle $\rb N_{\tilde f}'$ of $\rb N_{\tilde f}$
realises the real isotopy class $\rb L_{\tilde f}$ if and only if $a>b$.
\end{prop}

\begin{proof}
Recall that the isotopy class $\rb L_{\tilde f}$ is constructed as follows:
in the local holomorphic coordinates on $N_{\tilde f}$ defining a
real holomorphic splitting $N_{\tilde f}=\ol_{\cb P^1}(2d-k-1)\oplus\ol_{\cb P^1}(2d-k-1)$,
the isotopy class of $\rb L_{\tilde f}$ is the isotopy class of the real part of the degree
$2d-k-2$ line subbundle of $N_{\tilde f}$ whose real fibers rotate positively
with respect to the positive basis of $\rb N_{\tilde f}$.
Note that there are two holomorphic real line subbundles of $N_{\tilde f}$ of degree $2d-k-2$
depending on whether the real part of a fiber rotates in $\rb N_{\tilde f}$ positively.
Since $N_{\tilde f}'=N_{f}'\otimes\ol_{\cb P^1}(-z)$ and
$N_{\tilde f}=N_{f}\otimes\ol_{\cb P^1}(-z)$, we can get
that the fibers of $\rb N_{f}'$ rotate positively with respect to the positive basis of $\rb N_{f}$
if and only if the fibers of $\rb N_{\tilde f}'$ rotate positively with
respect to the positive basis of $\rb N_{\tilde f}$.
From \cite[Proposition 4.4]{bg2016}, we know the holomorphic real line subbundle
$\rb N_{f}'$ of $\rb N_{f}$ realises the real isotopy class $\rb L_{f}$ if and only if $a>b$.
\end{proof}

\begin{lem}\label{subsec3.1-lem-2}
Let $f:\cb P^1\to \cb P^3$ be a real algebraic immersion, whose image is contained in $Q$,
passing through $x_0$ such that it has bidegree $(a,b)$ in the positive basis with $a\neq b$.
Suppose $z=f^{-1}(x_0)\subset\cb P^1$ contains $k$ points, and $k$ is even.
Let $\tilde f$ be the strict transform of $f$ in $Y$.
Then
$$
s_Y(\rb N_{\tilde f}')=m_{\wt Q}(\tilde f(\cb P^1))+b+\frac{k}2.
$$
\end{lem}

\begin{proof}
The idea of the proof is similar to the proofs of
\cite[Lemma 2.1, Corollary 4.3]{bg2016} and \cite[Theorem 4.4]{wel2005b}.
Let $(\sigma_1(x),\sigma_2(x),\sigma_3(x))_{x\in\tilde f(\rb P^1)}$ be a loop in the
$SO_3(\rb)$-principal bundle $R_Y$ of orthonormal frames of $T\rb Y$ constructed in
Section $\ref{subsec2.2}$. Fix an orientation of $\rb P^1$ which induces an orientation
of the normal bundle $\rb N_{\tilde f}$.
Assume $\tilde f(\rb P^1)\cap\rb E=\{p_1,p_2,\ldots,p_r\}$.
Since $k$ is even, $r$ must be even too. Construct a loop
$(\tilde\sigma_1(x),\tilde\sigma_2(x),\tilde\sigma_3(x))_{x\in\tilde f(\rb P^1)}$
of $R_Y$ as the concatenation of the $2r$ following paths.
For the convenience, let $p_0=p_r$ be the same point.
When $i=0$, first, choose $(\sigma_1(x),\sigma_2(x),\sigma_3(x))_{x\in[p_{2i},p_{2i+1)}]}$.
Then construct a path completely included in the fiber of $R_Y$ over $p_{2i+1}$.
This path is obtained from $(\sigma_1(p_{2i+1}),\sigma_2(p_{2i+1}),\sigma_3(p_{2i+1}))$ by having this
frame turning of half a twist in the positive direction around the axis generated by $\sigma_1(p_{2i+1})$.
Note that the end point of this path is the frame $(\sigma_1(p_{2i+1}),-\sigma_2(p_{2i+1}),-\sigma_3(p_{2i+1}))$.
Next, the piece is chosen to be the path $(\sigma_1(x),-\sigma_2(x),-\sigma_3(x))_{x\in[p_{2i+1},p_{2i+2}]}$.
Last, we construct a path completely included in the fiber of $R_Y$ over $p_{2i+2}$.
This piece is obtained from $(\sigma_1(p_{2i+2}),-\sigma_2(p_{2i+2}),-\sigma_3(p_{2i+2}))$ by
having this frame turning of half a twist in the negative direction around the axis generated by $\sigma_1(p_{2i+2})$.
Note that the end point of this frame is $(\sigma_1(p_{2i+2}),\sigma_2(p_{2i+2}),\sigma_3(p_{2i+2}))$.
For $i=1,\ldots,\frac{r}2-1$, we just repeat the above construction. Finally,
we constructed a loop in $R_Y$ which is homotopic to
$(\sigma_1(x),\sigma_2(x),\sigma_3(x))_{x\in\tilde f(\rb P^1)}$.
$\rb Y=\rb P^3\#\overline{\rb P}^3$ is the connected sum of $\rb P^3$
and $\overline{\rb P}^3$ in a neighborhood of $x_0$.
Let the radius of the ball used to construct the connected sum
converge to zero, then the curve $\tilde f(\rb P^1)$ will degenerate to the union of the curve
$f(\rb P^1)$ and $r$ lines $\rb D_1$,...,$\rb D_{r}$ of $\overline{\rb P}^3$ passing through $x_0$.
The loop $(\tilde\sigma_1(x),\tilde\sigma_2(x),\tilde\sigma_3(x))_{x\in\tilde f(\rb P^1)}$ degenerates
to the union of a loop $(\tilde\sigma_1,\tilde\sigma_2,\tilde\sigma_3)_{f(\rb P^1)}$ of $R_{\cb P^3}$
and $r$ loops $(\tilde\sigma_1,\tilde\sigma_2,\tilde\sigma_3)_{\rb D_1}$, ...,
$(\tilde\sigma_1,\tilde\sigma_2,\tilde\sigma_3)_{\rb D_{r}}$ of $\overline R_{\cb P^3}$.
From the above construction, we know the loops
$(\tilde\sigma_1,\tilde\sigma_2,\tilde\sigma_3)_{\rb D_{2i}}$ differ from
$(\tilde\sigma_1,\tilde\sigma_2,\tilde\sigma_3)_{\rb D_{2i+1}}$ by a non-trivial loop in a
fiber of $\overline P_{\cb P^3}$. So
$s_{\overline{\cb P}^3}(\rb L_{\rb D_{2i}})+s_{\overline{\cb P}^3}(\rb L_{\rb D_{2i+1}})$ is odd.
Therefore,
$$
\Sigma_{i=1}^{r}s_{\overline{\cb P}^3}(\rb L_{\rb D_i})=\frac{r}{2}\mod2.
$$
\begin{figure}[h!]
\begin{center}
\begin{tabular}{ccc}
\includegraphics[width=3cm, angle=0]{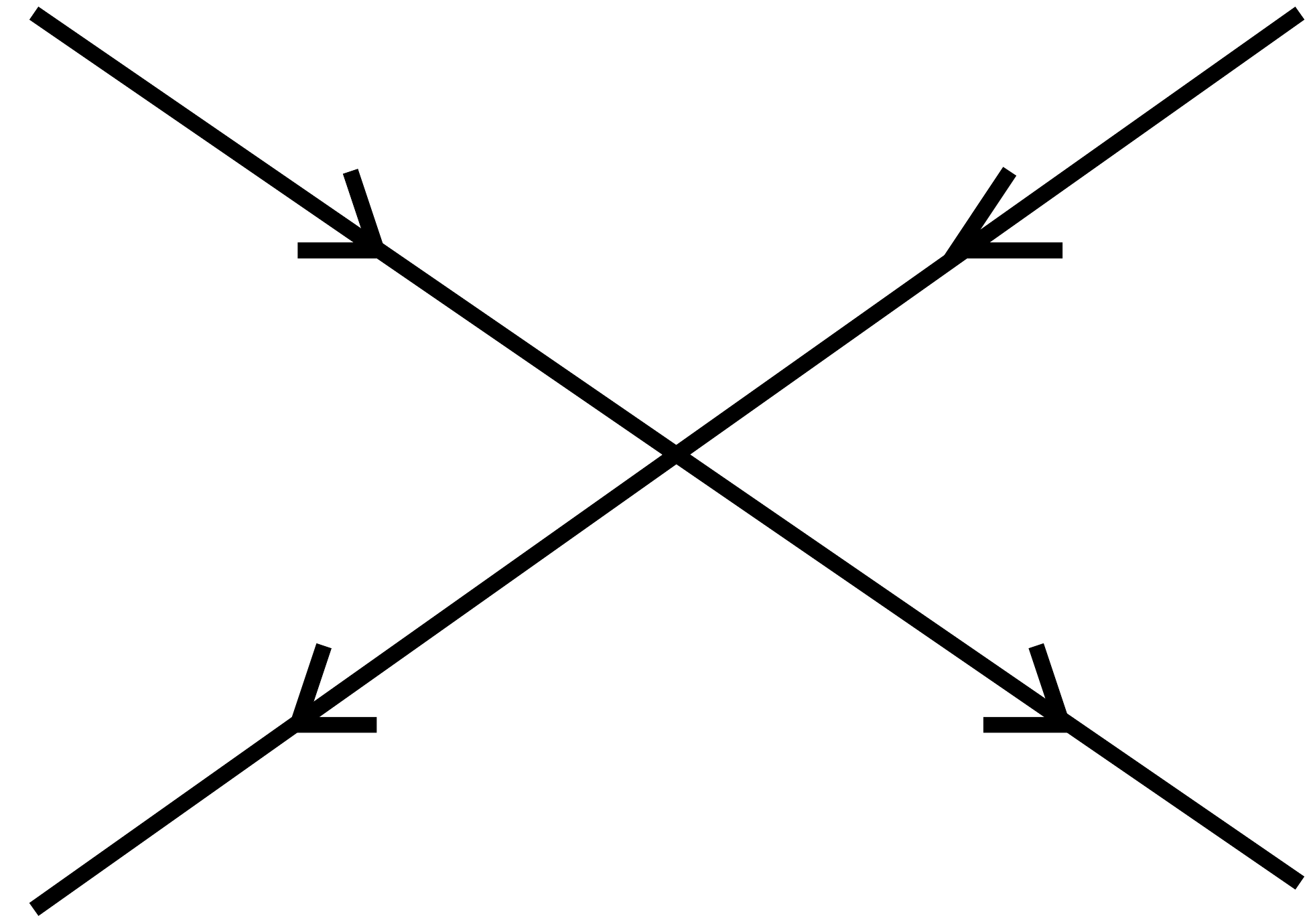}
&
\hspace{4ex}
\put(0, 30){$\dashrightarrow$}
\hspace{4ex}
&
\includegraphics[width=3cm, angle=0]{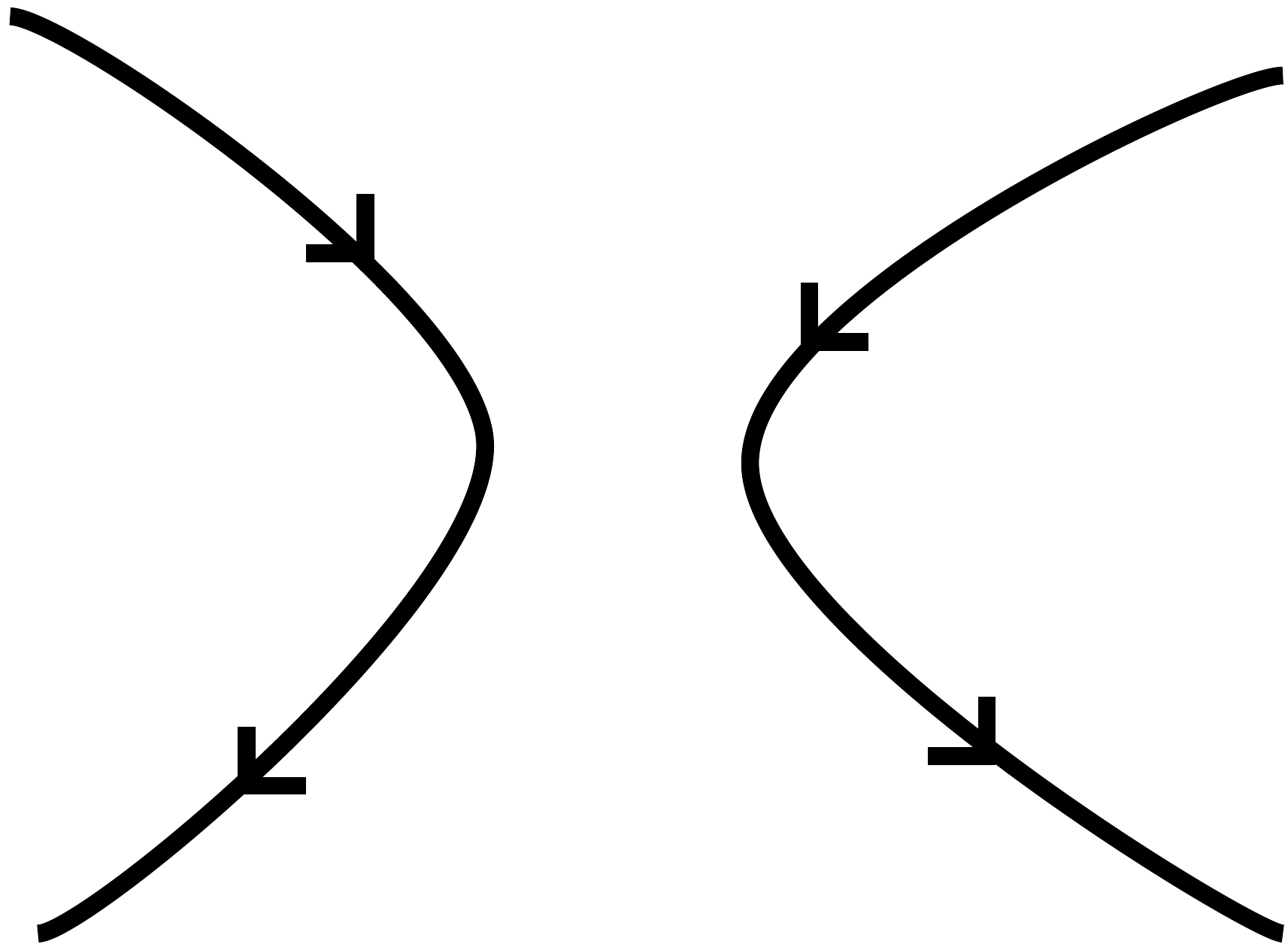}
\end{tabular}
\end{center}
\caption{Smoothing the node}
\label{figure1}
\end{figure}
For the loop $(\tilde\sigma_1,\tilde\sigma_2,\tilde\sigma_3)_{f(\rb P^1)}$,
we first smooth each node of $f(\rb P^1)$ as Figure $\ref{figure1}$, then smooth the
order $r$ singular point of $f(\rb P^1)$ at $x_0$ as Figure $\ref{figure2}$.
\begin{figure}[h!]
\begin{center}
\begin{tabular}{ccccccc}
\includegraphics[width=2.5cm, angle=0]{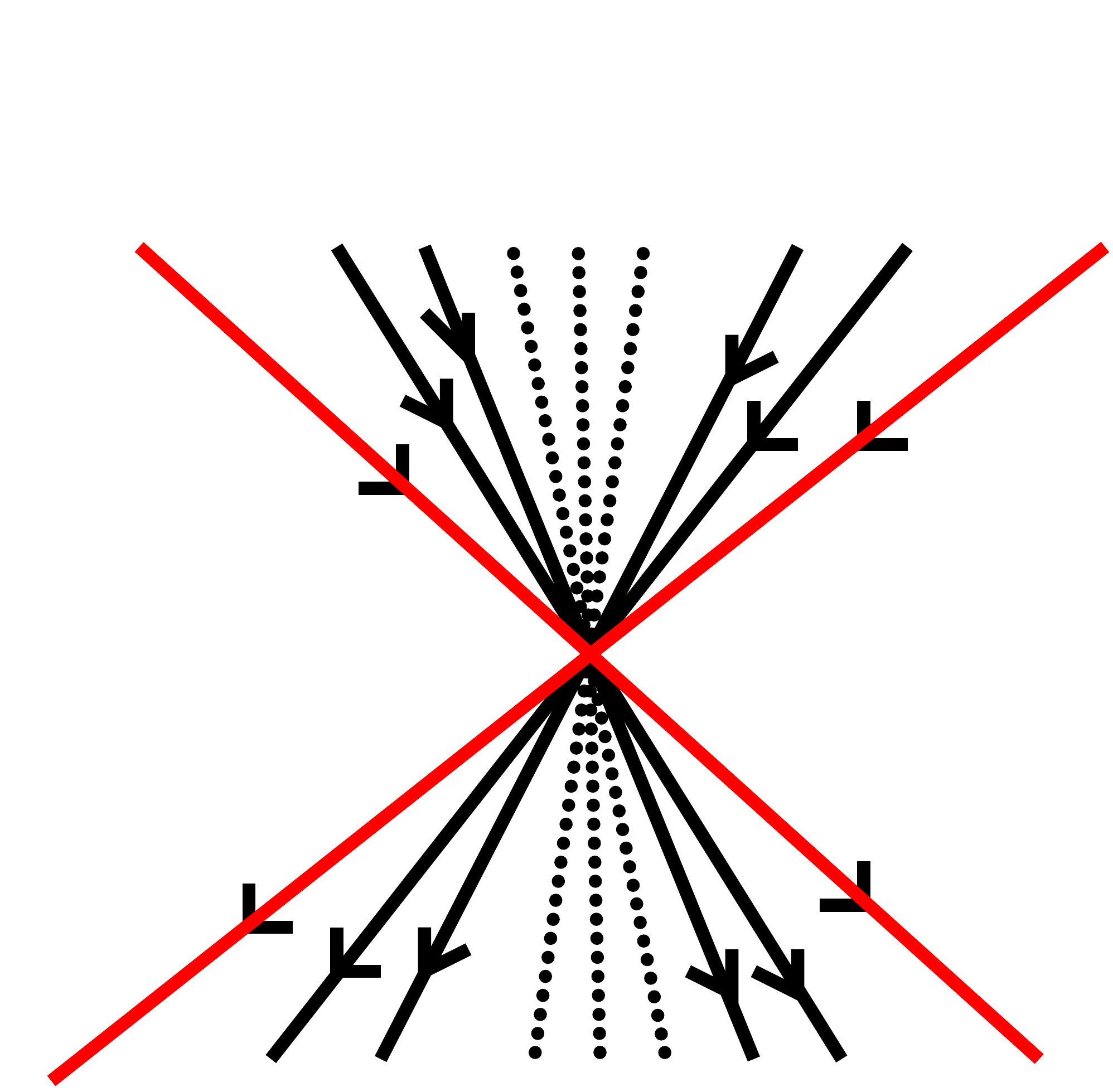}
&
\hspace{0ex}
\put(0, 25){$\dashrightarrow$}
\hspace{0ex}
&
\includegraphics[width=2.5cm, angle=0]{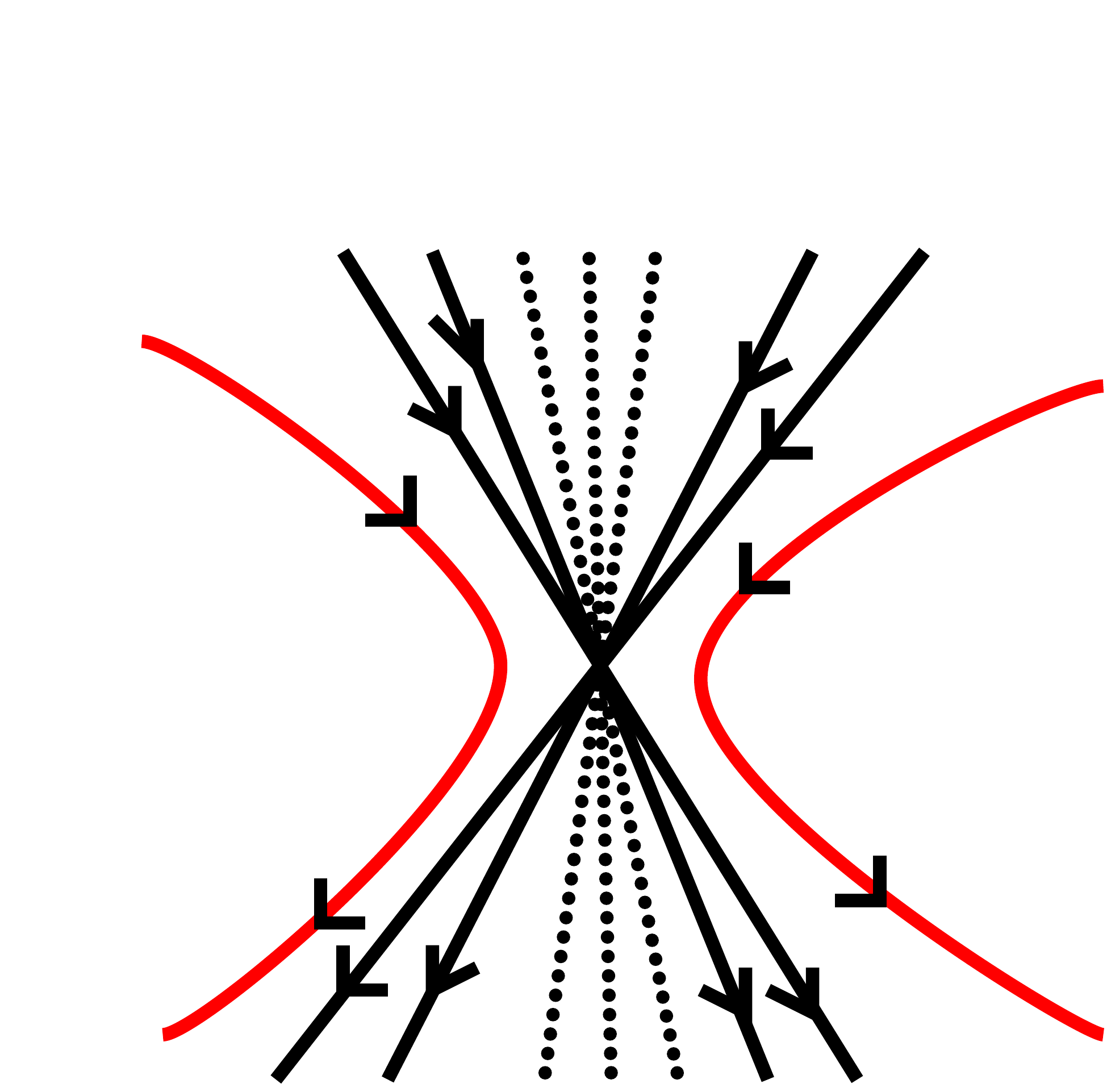}
&
\hspace{0ex}
\put(0, 25){$\dashrightarrow$}
\hspace{0ex}
&
\includegraphics[width=2.5cm, angle=0]{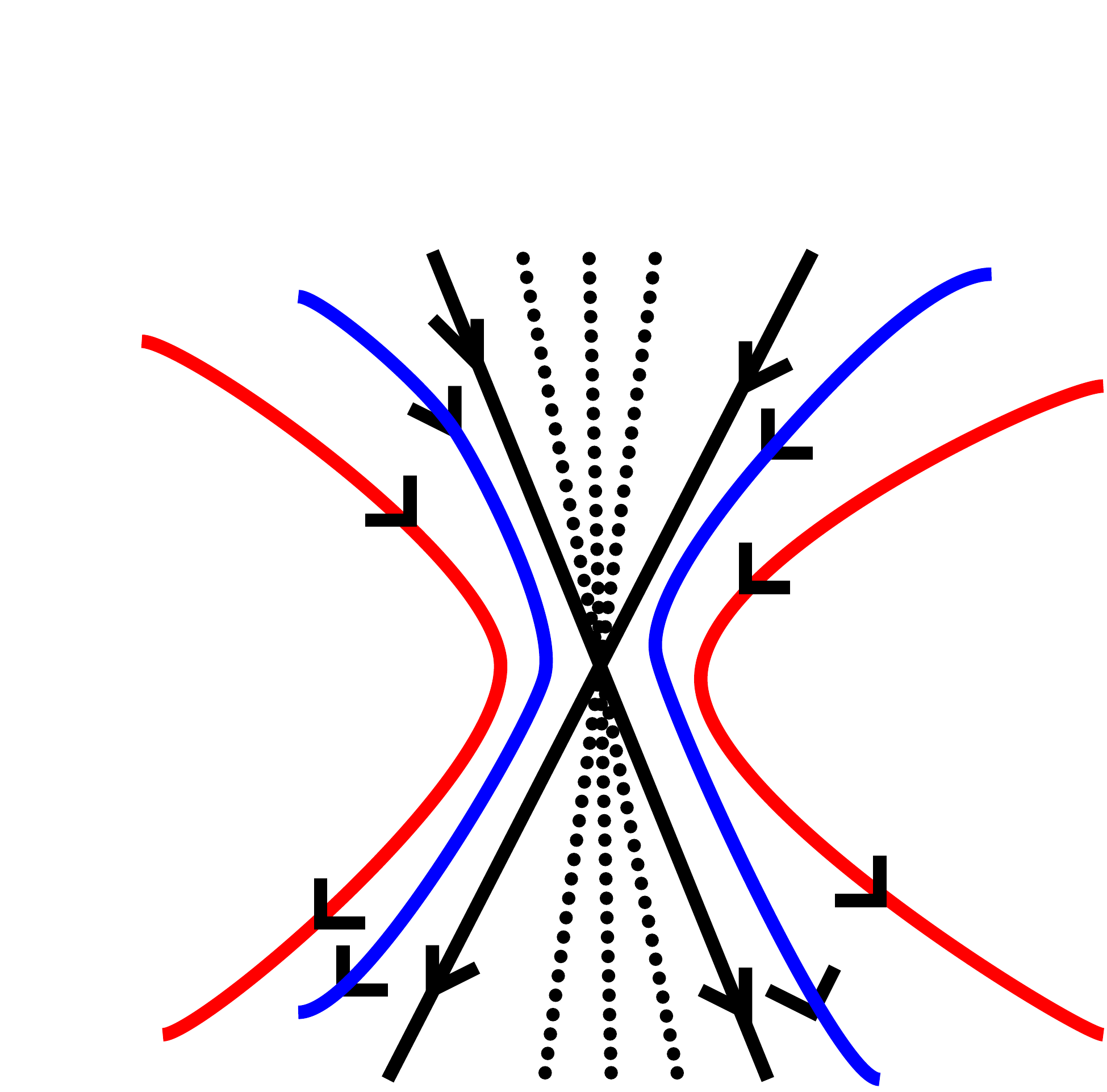}
&
\hspace{0ex}
\put(0, 25){$\dashrightarrow$}
\hspace{0ex}
&
\includegraphics[width=2.5cm, angle=0]{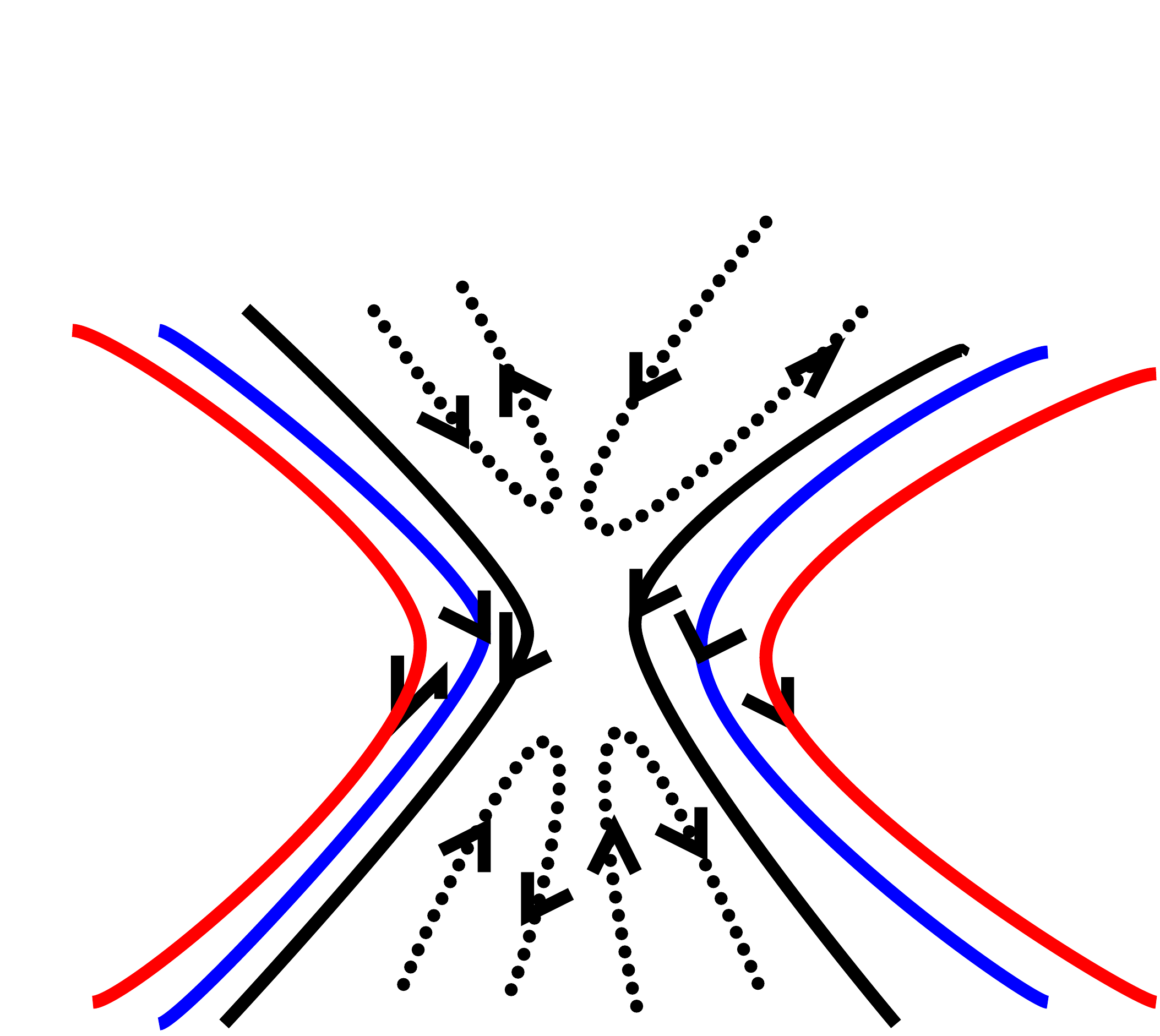}
\\
\footnotesize{(a). order $r$}
&
\hspace{0ex}
&
\footnotesize{(b). order $r-2$}
&
\hspace{0ex}
&
\footnotesize{(c). order $r-4$}
&
\hspace{0ex}
&
\footnotesize{(d). smoothing}

\end{tabular}
\end{center}
\caption{Smoothing the order $r$ singular point}
\label{figure2}
\end{figure}
We get a collection $\gamma$ of $n$ disjoint oriented circles embedded in $\rb P^1\times\rb P^1$.
And
$$
n=1+t+\frac{r}{2}\mod2,
$$
where $t$ is the number of hyperbolic nodes of $\tilde f(\cb P^1)$.
The loops of $\gamma$ do not intersect with each other, so the non-trivial loops
can only realize the same class $p[\rb l_1]+q[\rb l_2]$ with $p$ and $q$ relatively prime.
The loop realising the trivial class defines a non-trivial loop in $\pi_1(SO_3(\rb))$,
and the non-trivial loop defines $q$ times the non-trivial loop in $\pi_1(SO_3(\rb))$,
because $(l_1,l_2)$ is a positive basis.
Suppose $f(\rb P^1)=mp[\rb l_1]+mq[\rb l_2]$.
So $s_{\cb P^3}(\rb N_f')=n-m+qm\mod2$.
It means
$$
s_{\cb P^3}(\rb N_f')=t+1+\frac{r}{2}-m+qm\mod2.
$$
From the adjunction formula, $\tilde f(\cb P^1)$ has exactly $(a-1)(b-1)-\frac{k^2}2+\frac{k}2$ nodes
which contains $t$ hyperbolic nodes and $m_{\wt Q}(\tilde f(\cb P^1))$ elliptic nodes.
Thus,
\begin{equation*}
\begin{aligned}
s_{\cb P^3}(\rb N_f')=&(a-1)(b-1)-\frac{k^2}{2}+\frac{k}{2}\\
&+m_{\wt Q}(\tilde f(\cb P^1))+1+\frac{r}{2}-m+qm\mod2.
\end{aligned}
\end{equation*}
Note that $k$ is even and $(mp,mq)=(a,b)\mod2$.
So $m$ is even if and only if both $a$ and $b$ are even.
Hence
$$
s_{\cb P^3}(\rb N_f')=\frac{k}{2}+m_{\wt Q}(\tilde f(\cb P^1))+\frac{r}{2}+b\mod2.
$$
Since
$s_Y(\rb N_{\tilde f}')=s_{\cb P^3}(\rb N_f')+\Sigma_{i=1}^{r}s_{\overline{\cb P}^3}(\rb L_{\rb D_i})$
and $r$ is even, we get
$$
s_Y(\rb N_{\tilde f}')=m_{\wt Q}(\tilde f(\cb P^1))+b+\frac{k}2.
$$
\end{proof}

\section{Proof of the main results}

Denote by $\ml^*_{0,n}(X,d)$ the space of genus $0$ stable maps
$$
f:(\cb P^1,x_1,\ldots,x_n)\to X
$$
from $\cb P^1$ with $n$ marked points whose images represent the class $d$ in $X$.
Note that the stable maps are considered up to reparametrization.
There is a natural map:
\begin{equation*}
\begin{aligned}
ev:\ml^*_{0,n}(X,d)&\to ~~~X^n\\
f~&\mapsto(f(x_1),\ldots,f(x_n)),
\end{aligned}
\end{equation*}
which is called the evaluation map.
Let $\Omega_n\subset\wt C_4^n\subset(\cb P^3\#\overline{\cb P}^3\setminus E)^n$ be the
set of configurations of $n$ distinct points on $\wt C_4$.

Let $\wt Q\subset\wt\ql$ be the blow-up of a quadric,
$a\in\{0,\ldots,d\}$ and $\undl y\in\Omega_{2d-k-1}$. Denote by
$\overline\cl_{\wt Q}(\undl y,a,b,k)$ the set of stable maps $f:(C,x_1,\ldots,x_{2d-k-1})\to\wt Q$,
where $C$ is a connected nodal curve of arithmetic genus $0$,
such that $f(C)$ represents the class $(a,b,k)$ and $f(\{x_1,\ldots,x_{2d-k-1}\})=\undl y$.
\begin{prop}\label{subsec4.1-prop-1}
There exists a dense open subset $\Omega'_{2d-k-1}\subset\Omega_{2d-k-1}$
such that for any $a\in\{0,\ldots,d\}$, $\undl y\in\Omega'_{2d-k-1}$,
and for every stable map $(f:(C,x_1,\ldots,x_{2d-k-1})\to\wt Q)\in\overline\cl_{\wt Q}(\undl y,a,b,k)$,
we have:
\begin{itemize}
  \item $C$ is non-singular,
  \item $f$ is an immersion.
\end{itemize}
For any $\undl y\in\Omega'_{2d-k-1}$, $\wt Q$ contains exactly
%$GW_{0,2d-k-1}(\wt{Q},(a,d-a,k))$
$\langle[pt],\ldots,[pt]\rangle^{\wt Q}_{0,(a,b,k)}$ rational curves of class
$(a,d-a,k)$ and passing through $\undl y$.
\end{prop}

\begin{proof}
Let $\hl$ be the set of irreducible nodal rational curves in $\wt Q$ of class $(a,d-a,k)$.
By the proof of Lemma $\ref{subsec4.1-lem-1}$,
we know $H^1(\cb P^1,f^*T\wt Q)=0$. Therefore,
$\hl$ is a quasi-projective subvariety of dimension $2d-k-1$
of the linear system $|al_1+(d-a)l_2-kE_{\wt Q}|$.
Let $\overline\hl$ be the Zariski closure of $\hl$ and
$\fl=|(al_1+(d-a)l_2-kE_{\wt Q})|_{\wt C_4}|$. In $\wt Q$,
$\wt C_4$ represents class $2[l_1]+2[l_2]-[E_{\wt Q}]$, so the degree
of linear system $\fl$ is $2d-k$. As the genus of $\wt C_4$ is $1$,
every element of $\fl$ is determined by its $2d-k-1$ points.
Therefore, an element $\undl y$ of $\Omega_{2d-k-1}$ induces an element
$[\undl y]$ of $\fl$. $\wt C_4$ is not a component of any curve in $\overline\hl$,
we can define a map
\begin{equation*}
\begin{aligned}
\phi:\overline\hl&\to~\fl\\
C&\mapsto C\cap\wt C_4.
\end{aligned}
\end{equation*}
Note that $\dim\phi(\overline\hl\setminus\hl)\leq2d-k-2$.
The image of every element $f\in\overline\cl_{\wt Q}(\undl y,a,b,k)$,
with $\undl y\in\Omega_{2d-k-1}$
such that $[\undl y]\notin\phi(\overline\hl\setminus\hl)$,
is a nodal irreducible rational curve, $i.e.$ $C=\cb P^1$ and $f$ is an immersion.
The rest of the proof follows from Lemma $\ref{subsec4.1-lem-1}$.
\end{proof}

\begin{lem}\label{subsec4.1-lem-1}
A map $f\in\ml_{0,2d-k-1}^*(\wt Q,(a,d-a,k))$ is regular for the
corresponding evaluation map $ev$ if and only if it is an immersion.
\end{lem}

\begin{proof}
From the proof of \cite[Lemma 1.3]{wel2005b}, we know the cokernel of $d_{(f,\undl y)}ev$
is identified with the cokernel of the following composition:
\begin{equation}\label{subsec4.1-eq-1}
\begin{aligned}
H^0(\cb P^1;f^*T\wt Q)~&\to H^0(\cb P^1;\nl_f)&&\to H^0(\cb P^1;\nl_f|_{\undl y})\\
v~~~&\to~~~[v]&&\to~~[v(\undl y)],
\end{aligned}
\end{equation}
where $\nl_f$ is the quotient sheaf,
and the first morphism of equation $(\ref{subsec4.1-eq-1})$ is surjective.
Let $\nl_{f,-\undl y}=\nl_f\otimes\ol_{\cb P^1}(-\undl y)$.
The short exact sequence of sheaves $0\to\nl_{f,-\undl y}\to\nl_f\to\nl_f|_{\undl y}\to0$
implies the long exact sequence
\begin{equation}\label{subsec4.1-eq-2}
\begin{aligned}
\cdots&\to H^0(\cb P^1;\nl_f)\to H^0(\cb P^1;\nl_f|_{\undl y})\to H^1(\cb P^1;\nl_{f,-\undl y})\\
&\to H^1(\cb P^1;\nl_f)\to0\to\cdots.
\end{aligned}
\end{equation}
If $f$ is an immersion, we know $\nl_f\simeq\ol_{\cb P^1}(2d-k-2)$ and $H^1(\cb P^1;\nl_f)=0$.
Actually, from the exact sequence $0\to T\cb P^1\to f^*T\wt Q\to\nl_f\to0$,
we also can get $H^1(\cb P^1;f^*T\wt Q)=0$.
So the cokernel of
$$
d_{(f,\undl y)}ev=H^1(\cb P^1;\nl_{f,-\undl y})=H^1(\cb P^1;\ol_{\cb P^1}(-1))=0.
$$
Now suppose $f$ is a regular point of $d_{(f,\undl y)}ev$, then
$$
H^1(\cb P^1;\nl_{f,-\undl y})=H^1(\cb P^1;\nl_f).
$$
From the cohomology of $\cb P^1$,
we obtain $H^1(\cb P^1;\nl_{f,-\undl y})=H^1(\cb P^1;\nl_f)=0$. Therefore,
$\nl_f\simeq\ol_{\cb P^1}(a)$ with $a\geq2d-k-2$. By the adjunction formula, $a\leq2d-k-2$.
So $\nl_f\simeq\ol_{\cb P^1}(2d-k-2)$, and $f$ is an immersion.
\end{proof}

\begin{rem}\label{subsec4.1-rem-1}
%Let $\overline\cl'(\undl x)$ be the set of connected algebraic curves of arithmetic genus 0
%in $Y$ representing class $d[l]-k[e]$ and passing through $\undl x$.
By using \cite[Lemma $4.3$]{wel2005b}, one can prove,
similar to the proof of \cite[Proposition 5.3]{bg2016}, that the regular values of
$$
ev:\ml^*_{0,2d-k}(\cb P^3\#\overline{\cb P}^3,d[l]-k[e])\to(\cb P^3\#\overline{\cb P}^3)^{2d-k}
$$
contained in $\Omega_{2d-k}$ form a dense open subset
$\Omega'_{2d-k}\subset\Omega_{2d-k}$ and the set
$\overline\cl'(\undl x)$, $\undl x\in\Omega'_{2d-k}$, contains exactly
$\langle[pt],\ldots,[pt]\rangle^{\cb P^3\#\overline{\cb P}^3}_{0,d[l]-k[e]}$
%$:=GW_{0,2d-k}(\cb P^3\#\overline{\cb P}^3,d-k[E])$
elements.
\end{rem}

\begin{proof}[Proof of Theorem $\ref{thm2}$]
Choose $\undl x\in\Omega'_{2d-k}$, where $\Omega'_{2d-k}$
is the set defined in Remark $\ref{subsec4.1-rem-1}$,
and $\undl y\subset\undl x$ a set consisting of $2d-k-1$ points.
Denote by $\ml_a$ the set of blow-up of quadrics in $\wt \ql$ corresponding to
a solution of equation $(\ref{subsec3.1-eq-1})$ which consists of $(d-2a)^2$ elements.
From Lemma $\ref{subsec3.1-lem-1}$ , we have
$$
|\overline\cl'(\undl x)|=\sum_{0\leq a< b}\sum_{\wt Q\in\ml_a}|\overline\cl_{\wt Q}(\undl y,a,b,k)|.
$$
According to Remark $\ref{subsec4.1-rem-1}$,
$|\overline\cl'(\undl x)|=\langle[pt],\ldots,[pt]\rangle^{\cb P^3\#\overline{\cb P^3}}_{0,d[l]-k[e]}$.
We can get every map $(f:\cb P^1\to\cb P^3\#\overline{\cb P}^3)\in\overline\cl'(\undl x)$
is an immersion with the help of \cite[Lemma 1.3, Lemma 4.3]{wel2005a},
for $\undl x$ is a regular value of the corresponding evaluation map $ev$.
By Proposition $\ref{subsec4.1-prop-1}$, we have
$|\overline\cl_{\wt Q}(\undl y,a,b,k)|=\langle[pt],\ldots,[pt]\rangle^{\wt Q}_{0,(a,b,k)}$.
\end{proof}

\begin{proof}[Proof of Theorem $\ref{thm1}$]
Let $C_4$ be a real elliptic curve of degree 4 in $\cb P^3$ and $x_0\in\rb C_4$.
Denote by $Y$ the blow-up of $\cb P^3$ at $x_0$, and by $\wt C_4$ the strict transform of $C_4$.
Let $d\in\zb_{>0}-2\zb$, $k\in2\zb_{\geq0}$ with $d\geq k$.
Choose a real configuration $\undl x\in\Omega'_{2d-k}$ of $2d-k$ points which
contains at least one real point.
From Lemma $\ref{subsec3.1-lem-1}$ , we know the image of every real map
$(f:\cb P^1\to\cb P^3\#\overline{\cb P}^3)\in\overline\cl'(\undl x)$ represents class
$(a,d-a,k)$ or $(d-a,a,k)$, where $a\in\{0,...,\frac{d-1}{2}\}$,
in the positive basis of the blow-up of quadric $\wt Q$ containing $f(\cb P^1)$.
From Proposition $\ref{subsec3.1-prop-2}$ and Lemma $\ref{subsec3.1-lem-2}$,
we can obtain
$$
s_Y(f(\cb P^1))=
\left\{\begin{array}{ll}
m_{\wt Q}(f(\cb P^1))+\frac{k}2, & \mbox{if $a$ is even,} \\
m_{\wt Q}(f(\cb P^1))+\frac{k}2+1, & \mbox{if $a$ is odd.}
\end{array}
\right.
$$
The sum of spinor states of real elements of $\overline\cl'(\undl x)$ whose images are contained in $\wt Q$ is
$(-1)^{\frac{k}2}W_{\wt Q}((a,b,k),s)$ if $a$ is even and
$(-1)^{1+\frac{k}2}W_{\wt Q}((a,b,k),s)$ if $a$ is odd.
Since equation $(\ref{subsec3.1-eq-1})$ always has a real solution and
two real solutions differ by a real torsion element of order $d-2a$, we know
equation $(\ref{subsec3.1-eq-1})$ has $d-2a$ real solutions for every $a\in\{0,\ldots,\frac{d-1}{2}\}$.
Then one can complete the rest part of the proof with
the same argument of the proof of Theorem $\ref{thm2}$.
\end{proof}

\begin{rem}\label{subsec4.1-rem-2}
The vanishing of $W_{\cb P^3\#\overline{\cb P}^3}(d[l]-k[e],s)$ for positive even $d,k$ with $d\geq k$
and $s\leq d-\frac{k}2-1$ can also be proved in this way.
If $\tilde h$ and $e_{\widetilde Q}$ are on the same component of $\rb\wt C_4$ with $p$ and $\undl x$ is on the
component which does not contain $p$, then equation $(\ref{subsec3.1-eq-1})$ has no real solution.
\end{rem}

%\bibliographystyle{alpha}
%\bibliography{sg}

\end{document}